\pdfoutput=1
\RequirePackage{ifpdf}
\ifpdf 
\documentclass[pdftex]{sigma}
\else
\documentclass{sigma}
\fi

\usepackage{mathtools,braket,enumerate}
\usepackage{tikz}
\usetikzlibrary{decorations}
\usetikzlibrary{arrows}
\tikzstyle{v} = [circle, draw, inner sep=2pt, minimum size=3pt, fill=black]
\tikzstyle{e} = [fill, opacity=.2,
fill opacity=.2,
line join=round,
line width=13pt]
\pgfdeclarelayer{background}
\pgfsetlayers{background,main}

\pgfdeclaredecoration{dashsoliddouble}{initial}{
 \state{initial}[width=\pgfdecoratedinputsegmentlength]{
 \pgfmathsetlengthmacro\lw{.3pt+.5\pgflinewidth}
 \begin{pgfscope}
 \pgfpathmoveto{\pgfpoint{0pt}{\lw}}%
 \pgfpathlineto{\pgfpoint{\pgfdecoratedinputsegmentlength}{\lw}}%
 \pgfmathtruncatemacro\dashnum{%
 round((\pgfdecoratedinputsegmentlength-3pt)/6pt)
 }
 \pgfmathsetmacro\dashscale{%
 \pgfdecoratedinputsegmentlength/(\dashnum*6pt + 3pt)
 }
 \pgfmathsetlengthmacro\dashunit{3pt*\dashscale}
 \pgfsetdash{{\dashunit}{\dashunit}}{0pt}
 \pgfusepath{stroke}
 \pgfsetdash{}{0pt}
 \pgfpathmoveto{\pgfpoint{0pt}{-\lw}}%
 \pgfpathlineto{\pgfpoint{\pgfdecoratedinputsegmentlength}{-\lw}}%
 \pgfusepath{stroke}
 \end{pgfscope}
 }
}

\numberwithin{equation}{section}

\newtheorem{Theorem}{Theorem}[section]

\newtheorem{Lemma}[Theorem]{Lemma}
\newtheorem{Proposition}[Theorem]{Proposition}
 { \theoremstyle{definition}
\newtheorem{Definition}[Theorem]{Definition}
\newtheorem{Example}[Theorem]{Example}
\newtheorem{Remark}[Theorem]{Remark}
\newtheorem{Question}[Theorem]{Question} }

\DeclareMathOperator*{\Der}{Der}
\DeclareMathOperator{\cl}{cl}
\DeclareMathOperator{\si}{si}
\DeclareMathOperator{\PG}{PG}

\begin{document}
\allowdisplaybreaks

\newcommand{\arXivNumber}{1908.01535}

\renewcommand{\thefootnote}{}

\renewcommand{\PaperNumber}{080}

\FirstPageHeading

\ShortArticleName{Modular Construction of Free Hyperplane Arrangements}

\ArticleName{Modular Construction \\ of Free Hyperplane Arrangements\footnote{This paper is a~contribution to the Special Issue on Primitive Forms and Related Topics in honor of Kyoji Saito for his 77th birthday. The full collection is available at \href{https://www.emis.de/journals/SIGMA/Saito.html}{https://www.emis.de/journals/SIGMA/Saito.html}}}

\Author{Shuhei TSUJIE}

\AuthorNameForHeading{S.~Tsujie}

\Address{Department of Education, Hokkaido University of Education, Hokkaido, Japan}
\Email{\href{mailto:tsujie.shuhei@a.hokkyodai.ac.jp}{tsujie.shuhei@a.hokkyodai.ac.jp}}
\URLaddress{\url{https://sites.google.com/view/tsujieshuheimath/}}

\ArticleDates{Received January 29, 2020, in final form August 13, 2020; Published online August 22, 2020}

\Abstract{In this article, we study freeness of hyperplane arrangements. One of the most investigated arrangement is a graphic arrangement. Stanley proved that a graphic arrangement is free if and only if the corresponding graph is chordal and Dirac showed that a graph is chordal if and only if the graph is obtained by ``gluing'' complete graphs. We will generalize Dirac's construction to simple matroids with modular joins introduced by Ziegler and show that every arrangement whose associated matroid is constructed in the manner mentioned above is divisionally free. Moreover, we apply the result to arrangements associated with gain graphs and arrangements over finite fields.}


\Classification{52C35; 05B35; 05C22; 13N15}

\renewcommand{\thefootnote}{\arabic{footnote}}
\setcounter{footnote}{0}

\section{Introduction}

A (central) \textit{hyperplane arrangement} $ \mathcal{A} $ over a field $ \mathbb{K} $ is a finite collection of subspaces of codimension $ 1 $ in a finite dimensional vector space $ \mathbb{K}^{\ell} $.
A standard reference for arrangements is~\cite{orlik1992arrangements}.
Let $ S $ denote the polynomial algebra $ \mathbb{K}[x_{1}, \dots, x_{\ell}] $, where $ (x_{1}, \dots, x_{\ell}) $ is a basis for the dual space $ \big(\mathbb{K}^{\ell}\big) ^{\ast} $.
Let $ \Der(S) $ denote the \textit{module of derivations} of $ S $, that is,
\begin{gather*}
\Der(S) \coloneqq \{ \theta \colon S \to S \,|\, \theta \text{ is $ S $-linear and } \theta(fg) = f\theta(g) + \theta(f)g \text{ for any } f,g \in S \}.
\end{gather*}
The \textit{module of logarithmic derivations} $ D(\mathcal{A}) $ is defined by
\begin{gather*}
D(\mathcal{A}) \coloneqq \{ \theta \in \Der(S) \,|\, \theta(\alpha_{H}) \in \alpha_{H}S \text{ for all } H \in \mathcal{A} \},
\end{gather*}
where $ \alpha_{H} $ is a linear form such that $ \ker(\alpha_{H}) = H $.

\begin{Definition}
An arrangement $ \mathcal{A} $ is called \textit{free} if $ D(\mathcal{A}) $ is a free $ S $-module.
\end{Definition}

Although the definition of free arrangements is algebraic, Terao's celebrated factorization theorem \cite[Main Theorem]{terao1981generalized-im} shows a solid relation between algebra, combinatorics, and topology of arrangements.
Terao's conjecture asserts that the freeness of an arrangement is determined by its combinatorial property and it is still widely open.

One of typical family of arrangements is graphic arrangements.
Let $ \Gamma = ([n], E_{\Gamma}) $ denote a~simple graph, where $ [n] \coloneqq \{1, \dots, n\} $.
Define a \textit{graphic arrangement} $ \mathcal{A}(\Gamma) $ by
\begin{gather*}
\mathcal{A}(\Gamma) \coloneqq \{ \{x_{i}-x_{j}=0\} \,|\, \{i,j\} \in E_{\Gamma}\}.
\end{gather*}

A simple graph is \textit{chordal} if every cycle of length at least $ 4 $ has a chord, which is an edge connecting nonconsecutive vertices of the cycle.
Freeness of graphic arrangements is characterized in terms of graphs as follows.
\begin{Theorem}[{Stanley, see \cite[Theorem 3.3]{edelman1994free-mz}} for example]
A graphic arrangement $ \mathcal{A}(\Gamma) $ is free if and only if $ \Gamma $ if chordal.
\end{Theorem}

A vertex of a simple graph is called \textit{simplicial} if its neighborhood form a clique.
\begin{Theorem}[Dirac {\cite[Theorems 1, 2, and 4]{dirac1961rigid-aadmsduh}}]\label{Dirac}
The class of chordal graphs coincides with the smallest class $ \mathcal{C} $ of graphs satisfying the following conditions.
\begin{enumerate}[$(i)$]\itemsep=0pt
\item\label{chordal 1} The null graph is belongs to $ \mathcal{C} $.
\item\label{chordal 2} Suppose that a simple graph $ \Gamma $ has a simplicial vertex $ v $ and $ \Gamma \setminus v \in \mathcal{C} $.
Then $ \Gamma \in \mathcal{C} $.
\item\label{chordal 3} Let $ \Gamma $ be a simple graph on $ V = V_{1} \cup V_{2} $.
Suppose that the induced subgraph $ \Gamma[V_{1}\cap V_{2}] $ is complete (including the null graph) and $ E_{\Gamma} = E_{\Gamma[V_{1}]} \cup E_{\Gamma[V_{2}]} $.
If $ \Gamma[V_{1}] \in \mathcal{C} $ and $ \Gamma[V_{2}] \in \mathcal{C} $, then $ \Gamma \in \mathcal{C} $.
\end{enumerate}
\end{Theorem}

Note that every complete graph $ K_{n} $ belongs to $ \mathcal{C} $ by the condition (\ref{chordal 1}) and (\ref{chordal 2}), and the condition (\ref{chordal 2}) is unified with the condition~(\ref{chordal 3}) for non-complete graphs.
Thus a chordal graph is constructed by gluing complete graphs.

The purpose of this paper is to generalize the class consisting of chordal graphs in terms of matroids with conditions described in Theorem \ref{Dirac} and associate it with freeness of hyperplane arrangements.
See Oxley \cite{oxley2011matroid} for basic terminologies of matroids.

There are some generalizations of chordality for matroids in terms of circuits and chords (see~\cite{probert2018chordality}).
These are different from our generalization.

Let $ L $ be a geometric lattice.
An element $ X $ is called \textit{modular} if
\begin{gather*}
r(X) + r(Y) = r(X \wedge Y) + r(X \vee Y) \quad \text{ for all } Y \in L,
\end{gather*}
where $ r $ denotes the rank function of $ L $.
A flat $ X $ of a simple matroid $ M $ is called \textit{modular} if $ X $ is modular in $ L(M) $, the lattice of flats of~$ M $.

\begin{Definition}
A simple matroid $ M $ on the ground set $ E $ is said to be a \textit{modular join} if there exist two proper modular flats $ E_{1} $ and $ E_{2} $ of $ M $ such that $ E = E_{1} \cup E_{2} $.
We also say that $ M $ is the \textit{modular join over $ X $}, denoted $ M = P_{X}(E_{1}, E_{2}) = P_{X}(M_{1}, M_{2}) $, where $ M_{i} \coloneqq M|E_{i} $ for $ i = 1,2 $ and $ X \coloneqq E_{1} \cap E_{2} $.
\end{Definition}

\begin{Remark}
Ziegler \cite{ziegler1991binary-potams} introduced a modular join, which is a special case of a \textit{generalized parallel connection} or a \textit{strong join} investigated by Brylawski~\cite{brylawski1975modular-totams} and Lindstr{\"o}m~\cite{lindstrom1978strong-joctsa}.
Our definition of a modular join is different from Ziegler's one.
However, they are equivalent (see \cite[Propositioin~3.3]{ziegler1991binary-potams} and \cite[Proposition~5.10]{brylawski1975modular-totams}).
In addition, note that~$ X $ is modular in~$ M_{1}$,~$M_{2} $, and~$ M $.
\end{Remark}

\begin{Definition}A matroid $ M $ is called \textit{round} (or \textit{nonsplit}) if the ground set is not the union of two proper flats. A subset $ S $ of the ground set of $ M $ is \textit{round} if the restriction $ M|S $ is round.
\end{Definition}

It is well known that the graphic matroid of a simple graph without isolated vertices is round if and only if the graph is complete (see \cite[Theorem~4.2]{borissova2016regular-etpad} for example) and that any induced subgraph isomorphic to a complete graph corresponds to a modular flat of a graphic matroid (see \cite[Proposition~6.9.11 and below]{oxley2011matroid} for example).
Our generalization of chordal graphs is defined as follows.
\begin{Definition}\label{modularly extended matroid}
Let $ \mathcal{ME} $ be the minimal class of simple matroids which satisfies the following conditions.
\begin{enumerate}[$(i)$]\itemsep=0pt
\item The empty matroid is a member of $ \mathcal{ME} $.
\item If a simple matroid $ M $ has a modular coatom $ X $ and $ M|X \in \mathcal{ME} $, then $ M \in \mathcal{ME} $.
\item Let $ M $ be a modular join of $ M_{1} $ and $ M_{2} $ over a round flat.
If $ M_{1}, M_{2} \in \mathcal{ME} $, then $ M \in \mathcal{ME} $.
\end{enumerate}
We say that a simple matroid in $ \mathcal{ME} $ is \textit{modularly extended}.
\end{Definition}

Comparing the conditions in Theorem \ref{Dirac} and Definition \ref{modularly extended matroid}, we can prove that a simple graphic matroid is modularly extended if and only if the associated graph is chordal.

Recent studies \cite{suyama2019signed-dm} and \cite{torielli2020freeness-a} treat a similar class for signed graphs and their associated arrangements.
Moreover, we will study modularly extended matroids associated with gain graphs in Section~\ref{sec:applications}.

The linear dependence of an arrangement $ \mathcal{A} $ determines a simple matroid $ M(\mathcal{A}) $ on itself.
Namely, a subset $ \{H_{1}, \dots, H_{n}\} \subseteq \mathcal{A} $ is defined to be independent if the codimension of the intersection $ H_{1} \cap \dots \cap H_{n} $ is equal to $ n $.
The main theorem of this paper is as follows (see Definitions~\ref{definition divisional flag} and~\ref{definition divisional freeness} for the definition of divisional flags and divisional freeness).
\begin{Theorem}\label{main divisionally free}
Every modularly extended matroid has a divisional flag. In particular, if the linear dependence matroid $ M(\mathcal{A}) $ on an arrangement $ \mathcal{A} $ is modularly extended, then $ \mathcal{A} $ is divisionally free.
\end{Theorem}

A simple matroid is \textit{supersolvable} if it has a saturated chain consisting of modular flats, or equivalently it belongs to the minimal class satisfying (i) and (ii) in Definition \ref{modularly extended matroid}.
It is known that supersolvable arrangements are inductively free by \cite[Theorem 4.2]{jambu1984free-aim}, and hence divisionally free by \cite[Theorem~1.6]{abe2016divisionally-im}.
Clearly, the class $ \mathcal{ME} $ contains all supersolvable matroids.
Furthermore, we have the following theorem.
\begin{Theorem}\label{main ss}
The class $ \mathcal{ME} $ coincides with the minimal class which contains all supersolvable simple matroids and is closed under taking modular joins over round flats.
\end{Theorem}

A graphic arrangement is supersolvable if and only if the corresponding graph is chordal \cite[Proposition~2.8]{stanley1972supersolvable-au}. Therefore freeness and supersolvability are equivalent in the class of graphic arrangements.

As Ziegler \cite{ziegler1991binary-potams} mentioned, there exists a modular join of supersolvable matroids which is not supersolvable.
Therefore the class $ \mathcal{ME} $ is strictly larger than the class of supersolvable matroids.

In order to see an example, let $ M \in \mathcal{ME} $ be a non-supersolvable matroid which is a modular join $ M=P_{X}(M_{1},M_{2}) $ of supersolvable matroids $ M_{1} $ and $ M_{2} $.
The rank of $ M $ is computed by $ r(M) = r(M_{1}) + r(M_{2}) - r(X) $ by Brylawsky \cite[Propositioin~5.2]{brylawski1975modular-totams}.
Since $ M $ is not supersolvable, we have $ r(M_{i}) \leq r(M)-2 $ for $ i \in \{1,2\} $ and $ r(X) \geq 1 $.
Therefore $ r(M) \leq 2r(M)-5 $ and hence $ r(M) \geq 5 $.
Actually, there exists a non-supersolvable matroid in $ \mathcal{ME} $ whose rank is exactly $ 5 $ as follows.

\begin{Example}\label{example intro}
Let $ \mathcal{A}_{1} $ be an arrangement over $ \mathbb{R} $ consisting of the following $ 7 $ hyperplanes
\begin{gather*}
\{z=0\}, \, \{x_{1}=0\}, \, \{x_{2}=0\}, \, \{x_{1}-z=0\}, \, \{x_{2}-z=0\}, \, \{x_{1}-x_{2}=0\}, \, \{x_{1}+x_{2}=0\}.
\end{gather*}
Then $ \mathcal{A}_{1} $ is supersolvable with modular coatom $ \{x_{1}=x_{2}=0\} $.
Let $ \mathcal{A}_{2} $ be an isomorphic copy of $ \mathcal{A}_{1} $ and take a modular join of $ \mathcal{A}_{1} $ and $ \mathcal{A}_{2} $ over $ \{z=0\} $, that is, an arrangement $ \mathcal{A} $ consisting of the following $ 13 $ hyperplanes
\begin{gather*}
\{z=0\}, \, \{x_{1}=0\}, \, \{x_{2}=0\}, \, \{x_{1}-z=0\}, \, \{x_{2}-z=0\}, \, \{x_{1}-x_{2}=0\}, \, \{x_{1}+x_{2}=0\}, \\
\{y_{1}=0\}, \, \{y_{2}=0\}, \, \{y_{1}-z=0\}, \, \{y_{2}-z=0\}, \, \{y_{1}-y_{2}=0\}, \, \{y_{1}+y_{2}=0\}.
\end{gather*}
Then $ r(\mathcal{A})=5$, $M(\mathcal{A}) \in \mathcal{ME} $, and by Theorem \ref{main divisionally free}, $ \mathcal{A} $ is divisionally free. Furthermore, one can deduce that $ \mathcal{A} $ is not supersolvable by showing that it has no modular flats of rank $ 4 $ (or see Example~\ref{example frame}).
\end{Example}

The organization of this paper is as follows.
In Section~\ref{sec: preliminaries}, we recall basic properties about simple matroids and geometric lattices, including modularity.
In addition, we introduce divisional atoms and study them for modular joins. In Section~\ref{sec: proof}, we give a proof of Theorem~\ref{main divisionally free}. Finally, in Section~\ref{sec:applications}, we give applications to arrangements corresponding to gain graphs and arrangements over finite fields.

\section{Preliminaries}\label{sec: preliminaries}

\subsection{Simple matroids and geometric lattices}

In order to avoid confusion, we do not use geometrical terminology such as ``point" and ``hyperplane" for matroids and we call an element of a matroid an \textit{atom} and a flat of corank $ 1 $ of a matroid a \textit{coatom}.
This is lattice theoretic terminology due to the following well-known theorem.

\begin{Theorem}[see {\cite[p.~54, Theorem 2]{welsh2010matroid}} for example]\label{simple matroid geometric lattice correspondence}
The correspondence between a simple matroid and its lattice of flats is a bijection between simple matroids and geometric lattices.
\end{Theorem}

Thus any properties about simple matroids are translated into properties of geometric lattices, and vice versa.
For example, the contraction and the restriction of matroids are just intervals in the lattice of flats as follows.

\begin{Proposition}[{\cite[Proposition 3.3.8]{oxley2011matroid}}]\label{Oxley interval}
Let $ X $ be a flat of a matroid $ M $. Then
\begin{enumerate}[$(1)$]\itemsep=0pt
\item\label{Oxley interval 1} $ L(M/X) \simeq \big[X, \hat{1}\big] = \{ F \in L(M) \,|\, X \leq F \} $.
\item\label{Oxley interval 2} $ L(M|X) \simeq \big[\hat{0}, X\big] = \{ F \in L(M) \,|\, F \leq X \} $.
\end{enumerate}
\end{Proposition}

Note that the contraction of a simple matroid is not simple in general.
However, we can associate the simple matroid $ \si(M) $ with a matroid $ M $ and this operation does not affect the lattice of flats, that is, $ L(M) \simeq L(\si(M)) $.

If $ \mathcal{A} $ be an arrangement and $ M(\mathcal{A}) $ the simple matroid on $ \mathcal{A} $, then for any hyperplane $ H \in \mathcal{A} $, we have $ M\big(\mathcal{A}^{H}\big) \simeq \si(M(\mathcal{A})/H) $, where $ \mathcal{A}^{H} $ denotes the \textit{restriction} defined by
\begin{gather*}
\mathcal{A}^{H} \coloneqq \{ K \cap H \,|\, K \in \mathcal{A}\setminus\{H\}\}.
\end{gather*}

Note that the restriction of an arrangement does \emph{not} correspond to the restriction of a~matroid but the simplification of the contraction.
The restriction of a matroid corresponds to the localization of an arrangement.

\subsection{Characteristic polynomials}

Let $ L $ be a geometric lattice.
The \textit{characteristic polynomial} $ \chi(L,t) \in \mathbb{Z}[t] $ is defined by
\begin{gather*}
\chi(L,t) \coloneqq \sum_{X \in L}\mu(X)t^{r(\hat{1}) - r(X)},
\end{gather*}
where $ r $ denotes the rank function of $ L $ and $ \mu \colon L \to \mathbb{Z} $ denotes the one-variable M\"{o}bius function of $ L $ defined recursively by{\samepage
\begin{gather*}
\mu(X) \coloneqq \begin{cases}
 1 & \text{if } X = \hat{0}, \vspace{1mm}\\
-\displaystyle\sum_{Y < X}\mu(Y) & \text{otherwise.}
\end{cases}
\end{gather*}
The \textit{characteristic polynomial} of $ M $ is defined by $ \chi(M,t) \coloneqq \chi(L(M),t) $.}

The \textit{intersection lattice} $ L(\mathcal{A}) $ of a central arrangement $ \mathcal{A} $ is defined by
\begin{gather*}
L(\mathcal{A}) \coloneqq \bigg\{ \bigcap_{H \in \mathcal{B}}H \,\bigg|\, \mathcal{B} \subseteq \mathcal{A} \bigg\}
\end{gather*}
with a partial order by reverse inclusion.
Note that $ L(\mathcal{A}) $ is a geometric lattice and naturally isomorphic to $ L(M(\mathcal{A})) $.
The \textit{characteristic polynomial} $ \chi(\mathcal{A}, t) \in \mathbb{Z}[t] $ is defined by
\begin{gather*}
\chi(\mathcal{A},t) \coloneqq \sum_{X \in L(\mathcal{A})}\mu(X)t^{\dim X}.
\end{gather*}
Since the rank function of $ L(\mathcal{A}) $ is given by the codimension, $ \chi(\mathcal{A},t) = t^{\ell-r}\chi(M(\mathcal{A}), t) $, where $ \ell $ is the dimension of the ambient space and $ r $ is the rank of $ L(\mathcal{A}) $.
When $ \ell = r $ we say that $ \mathcal{A} $ is \textit{essential}.
It is well known that there exists an essential arrangement $ \mathcal{A}_{0} $ for every arrangement~$ \mathcal{A} $ such that $ L(\mathcal{A}) = L(\mathcal{A}_{0}) $ and $ \mathcal{A} $ is free if and only if~$ \mathcal{A}_{0} $ is free (see \cite{orlik1992arrangements} for details).

\subsection{Modularity}

We excerpt some conditions equivalent to modularity from Brylawski \cite{brylawski1975modular-totams}.
\begin{Proposition}[Brylawski {\cite[Theorem~3.3]{brylawski1975modular-totams}}]\label{brylawski modular equivalent}
Let $ X $ be an element of a geometric lattice~$ L $.
Then the following conditions are equivalent.
\begin{enumerate}[$(1)$]\itemsep=0pt
\item $ X $ is modular.
\item\label{brylawski modular equivalent 2} For $ Y \leq Z $ in $ L $, $ Y \vee (X \wedge Z) = (Y \vee X) \wedge Z $.
\item\label{brylawski modular equivalent 3} For all $ Y \in L $, $ [X \wedge Y, X] \simeq [Y, X \vee Y] $.
\item\label{brylawski modular equivalent 4} For any atom $ e \not\leq X $, $ [\hat{0}, X] \simeq [e, X \vee e] $ and $ X \vee e $ is modular in $ \big[e, \hat{1}\big] $.
\end{enumerate}
\end{Proposition}

\begin{Theorem}[Brylawski {\cite[Theorem 3.11]{brylawski1975modular-totams}} (the modular short-circuit axiom)]\label{Brylawsky modular short-circuit axiom}
Let $ M $ be a~simple matroid on the ground set $ E $ and $ X \subseteq E $ a nonempty subset.
Then $ X $ is a modular flat of $ M $ if and only if for every circuit $ C $ of $ M $ and an atom $ e \in C \setminus X $ there exist an atom $ x \in X $ and a circuit $ C^{\prime} $ such that $ e \in C^{\prime} \subseteq \{x\} \cup (C \setminus X) $.
\end{Theorem}

\begin{Theorem}[{Brylawski \cite[Corollary 3.4]{brylawski1975modular-totams}}]\label{Brylawski modular coatom}
A coatom $ X $ of a simple matroid on $ E $ is modular if and only if for any distinct two atoms $ e, e^{\prime} \in E\setminus X $ there exists $ e^{\prime\prime} \in X $ such that $ \{e,e^{\prime},e^{\prime\prime}\} $ forms a circuit.
\end{Theorem}

We give some properties of modularity required in this article.
\begin{Proposition}[{Jambu--Papadima \cite[Lemmas~1.3 and~1.9]{jambu1998generalization-t}}]\label{Jambu-Papadima}
Let $ X $ be a modular coatom of a simple matroid $ M $ on $ E $. Then for any two distinct atoms $a,b \in E\setminus X $, there exists unique $ f(a,b) \in X $ such that $ a,b,f(a,b) $ form a circuit.
Moreover, for any three distinct atoms $ a,b,c \in E\setminus X $, the atoms $ f(a,b)$, $f(a,c)$, $f(b,c) $ form a circuit.
\end{Proposition}

\begin{Proposition}[Brylawski {\cite[Proposition 3.5]{brylawski1975modular-totams}}]\label{Brylawski modular tower}
Let $ X $ be a modular flat of a simple matroid~$ M $ and~$ Y $ a modular flat of the restriction~$ M|X $, then~$ Y $ is a modular flat of~$ M $.
\end{Proposition}

\begin{Proposition}[Brylawski {\cite[Proposition~3.6]{brylawski1975modular-totams}}]\label{Brylawski modular intersection}
Let $ X $ and $ Y $ be modular flats of a simple matroid.
Then $ X \cap Y $ is a modular flat.
\end{Proposition}

\begin{Proposition}[Probert {\cite[Corollary 4.2.8]{probert2018chordality}}]\label{Probert}
Every modular flat of a round matroid is round.
\end{Proposition}

\begin{Theorem}[Stanley {\cite[Theorem 2]{stanley1971modular-au}}]\label{stanley division}
If $ X $ is a modular element of a geometric lattice $ L $, then $ \chi\big(\big[\hat{0},X\big],t\big) $ divides $ \chi(L,t) $.
\end{Theorem}

\subsection{Divisionality}

The following theorem plays an important role in this article.
\begin{Theorem}[Abe {\cite[Theorem 1.1 (division theorem)]{abe2016divisionally-im}}]\label{Abe division}
An arrangement $ \mathcal{A} $ is free if there exists $ H \in \mathcal{A} $ such that $ \mathcal{A}^{H} $ is free and $ \chi\big(\mathcal{A}^{H},t\big) $ divides $ \chi(\mathcal{A},t) $.
\end{Theorem}

Theorem \ref{Abe division} leads to the concepts of a divisional flag, which is originally defined for arrangements.
However, we define it for simple matroids as follows.

\begin{Definition}\label{definition divisional flag}
A \textit{divisional flag} of a simple matroid $ M $ of rank $ n $ is a sequence of flats \begin{align*}
\varnothing = X_{0} \subseteq X_{1} \subseteq \dots \subseteq X_{n} = E
\end{align*}
such that $ r(X_{i}) = i $ for $ i \in \{0, \dots, n \} $ and $ \chi(M/X_{i+1},t) \,|\, \chi(M/X_{i}, t) $ for $ i \in \{0, \dots, n-1 \} $.
\end{Definition}
Note that since $ \chi(M/X_{n},t)=1 $ and $ \chi(M/X_{n-1},t) = t-1 $ in the definition, the conditions $ \chi(M/X_{n},t) \,|\, \chi(M/X_{n-1},t) $ and $ \chi(M/X_{n-1},t) \,|\, \chi(M/X_{n-2},t) $ are always satisfied.

\begin{Definition}\label{definition divisional freeness}
An arrangement called \textit{divisionally free} if the matroid $ M(\mathcal{A}) $ has a divisional flag.
\end{Definition}
Note that, by Theorem \ref{Abe division}, every divisionally free arrangement is free. In order to find a~divisional flag, define a divisional atom as follows.

\begin{Definition}An atom $ e $ of a simple matroid $ M $ is called \textit{divisional} if $ \chi(M/e,t) $ divides $ \chi(M,t) $.
\end{Definition}

\begin{Proposition}\label{divisional flag}
A nonempty simple matroid $ M $ has a divisional flag if and only if there exists a divisional atom $ e $ such that $ \si(M/e) $ has a divisional flag.
\end{Proposition}
\begin{proof}
Suppose that $ M $ has a divisional flag $ \varnothing = X_{0} \subseteq X_{1} \subseteq \dots \subseteq X_{n} = E $.
Then $ e \coloneqq X_{1} $ is a divisional atom and the images of $ X_{1}, \dots, X_{n} $ under the isomorphism $ [e, \hat{1}] \simeq L(\si(M/e)) $ given by Proposition \ref{Oxley interval}(\ref{Oxley interval 1}) form a divisional flag of $ \si(M/e) $.
The converse holds by a similar argument.
\end{proof}

Abe \cite[Proposition 5]{abe2017restrictions-pilt} proved that every supersolvable arrangement is divisionally free by constructing a divisional flag from a saturated chain of modular flats. The following lemma is a~generalization for simple matroids.

\begin{Lemma}\label{modular coatom divisional}
Let $ M $ be a simple matroid on the ground set $ E $ and $ X $ a modular coatom of $ M $.
Then every atom $ e \in E \setminus X $ is divisional and $ \si(M/e) \simeq M|X $.
In particular, every supersolvable matroid has a divisional flag.
\end{Lemma}
\begin{proof}
Take an atom $ e \in E\setminus X $, and $ e $ is a complement of $ X $, that is, $ X \wedge e = \hat{0} $ and $ X \vee e = \hat{1} $.
By Proposition \ref{brylawski modular equivalent}(\ref{brylawski modular equivalent 3}), we have $ \big[\hat{0}, X\big] \simeq \big[e, \hat{1}\big] $, which implies $ M|X \simeq \si(M/e) $.
Moreover, by Theo\-rem~\ref{stanley division}, the characteristic polynomial $ \chi(M/e,t) = \chi(\si(M/e),t) = \chi(M|X,t) = \chi\big(\big[\hat{0}, X\big], t\big) $ divides $ \chi(M,t) $ and hence $ e $ is divisional.

When $ M $ is supersolvable, $ \si(M/e) $ is supersolvable.
Therefore $ M $ has a divisional flag by induction and Proposition \ref{divisional flag}.
\end{proof}

\subsection{Modular joins}

We review some properties of modular joins and will show a relation between modular joins and divisional atoms.

\begin{Proposition}[see also Ziegler {\cite[Lemma 3.10]{ziegler1991binary-potams}}]\label{Ziegler case of modular coatom}
Let $ X $ be a minimal flat of a simple matroid $ M $ such that $ M $ is a modular join $ M = P_{X}(M_{1}, M_{2}) $ over $ X $.
If $ M_{1} $ has a modular coatom, then $ M_{1} $ has a divisional atom not belonging to $ X $.
\end{Proposition}
\begin{proof}
Let $ Z \subseteq E_{1} $ be a modular coatom of $ M_{1} $, where $ E_{1} $ denotes the ground set of $ M_{1} $.
By Lemma \ref{modular coatom divisional}, every element in $ E_{1}\setminus Z $ is a divisional atom of $ M_{1} $.
Assume that $ E_{1}\setminus Z \subseteq X $.
Then $ Z \cap X \subsetneq X $ and $ M $ is a modular join $ M = P_{Z \cap X}(M|Z, M_{2}) $ over $ Z \cap X $, which is a contradiction to the minimality of $ X $.
Hence $ E_{1}\setminus Z \not\subseteq X $ and every element $ E_{1}\setminus(Z \cup X) $ is a desired atom.
\end{proof}

\begin{Theorem}[Brylawski {\cite[Theorem 7.8]{brylawski1975modular-totams}}]\label{Brylawski decomposition}
Let $ M = P_{X}(M_{1}, M_{2}) $ be a modular join.
Then
\begin{align*}
\chi(M,t) = \frac{\chi(M_{1},t) \, \chi(M_{2},t)}{\chi(M|X,t)}.
\end{align*}
\end{Theorem}

The following proposition is essentially due to Brylawski for generalized parallel connections.
However, since we treat the special case of modular joins, we give a proof of the proposition below.
\begin{Proposition}[Brylawski {\cite[Theorem 5.11.4]{brylawski1975modular-totams}}]\label{Brylawski contraction}
Let $ M $ be a modular join $ M=P_{X}(M_{1},M_{2}) $ and $ e $ an atom of $ M_{1} $ not belonging to $ X $.
Then $ \si(M/e) $ is isomorphic to a modular join of $ \si(M_{1}/e) $ and $ M_{2}, $ and
\begin{align*}
\chi(\si(M/e),t) = \dfrac{\chi(\si(M_{1}/e),t) \chi(M_{2}, t)}{\chi(M|X, t)}.
\end{align*}
\end{Proposition}
\begin{proof}
Let $ E_{1} $ and $ E_{2} $ be the ground sets of $ M_{1} $ and $ M_{2} $, which are modular flats of~$ M $.
Take an atom $ e \in E_{1} \setminus X $.
The matroid $ \si(M/e) $ corresponds the interval $ \big[e, \hat{1}\big] $ of~$ L(M) $ under the correspondence mentioned in Proposition \ref{simple matroid geometric lattice correspondence}.
Note that $ E_{1} $ is modular in $ \big[e, \hat{1}\big] $ by Proposition~\ref{brylawski modular equivalent}(\ref{brylawski modular equivalent 2}) and $ E_{2} \vee e $ is modular in $ \big[e, \hat{1}\big] $ by Proposition~\ref{brylawski modular equivalent}(\ref{brylawski modular equivalent 4}).

The atoms of $ \si(M/e) $ are identified with the atoms of the interval $ \big[e, \hat{1}\big] $.
These atoms coincide with $ \{ e \vee e^{\prime} \,|\, e^{\prime} \in E\setminus \{e\} \} $, where $ E = E_{1} \cup E_{2} $ denotes the ground set of $ M $.
If $ e^{\prime} \in E_{1} $, then $ e \vee e^{\prime} \leq E_{1} $.
Suppose that $ e^{\prime} \in E_{2} $.
Then $ e \vee e^{\prime} \leq e \vee E_{2} $.
Thus $ \si(M/e) $ is a modular join of matroids corresponding to $ [e, E_{1}] $ and $ [e, E_{2} \vee e] $.
The matroid corresponding to $ [e, E_{1}] $ is isomorphic to $ \si(M_{1}/e) $.
By Proposition~\ref{brylawski modular equivalent}(\ref{brylawski modular equivalent 3}), $ [e, E_{2} \vee e] \simeq \big[\hat{0}, E_{2}\big] $. Hence the matroid corresponding to $ [e, E_{2} \vee e] $ is isomorphic to~$ M_{2} $.
Thus $ \si(M/e) $ is isomorphic to a modular join of $ \si(M_{1}/e) $ and~$ M_{2} $.

By Proposition \ref{brylawski modular equivalent}(\ref{brylawski modular equivalent 2}), $ E_{1} \wedge (E_{2} \vee e) = (E_{1} \wedge E_{2}) \vee e = X \vee e $.
Using Theorem \ref{Brylawski decomposition} and Proposition \ref{brylawski modular equivalent}(\ref{brylawski modular equivalent 3}), we have
\begin{align*}
\chi(\si(M/e),t)& = \dfrac{\chi([e, E_{1}],t) \, \chi([e, E_{2}\vee e], t)}{\chi([e,X \vee e]), t}
= \dfrac{\chi([e, E_{1}],t) \chi\big(\big[\hat{0}, E_{2}\big], t\big)}{\chi\big(\big[\hat{0},X\big]\big), t} \\
&= \dfrac{\chi(\si(M_{1}/e),t) \chi(M_{2}, t)}{\chi(M|X, t)}.\tag*{\qed}
\end{align*}\renewcommand{\qed}{}
\end{proof}

\begin{Lemma}\label{divisional atom of modular join}
Let $ M $ be a modular join $ M = P_{X}(M_{1},M_{2}) $.
Every divisional atom of $ M_{1} $ not belonging to $ X $ is a divisional atom of $ M $.
\end{Lemma}
\begin{proof}
Let $ e $ be a divisional atom of $ M_{1} $ such that $ e \not\in X $.
Then there exists an integer $ a $ such that $ \chi(M_{1},t) = (t-a)\chi(\si(M_{1}/e),t) $.
Using Proposition \ref{Brylawski contraction}, we have
\begin{align*}
\chi(M,t) = \dfrac{\chi(M_{1},t) \, \chi(M_{2},t)}{\chi(M|X,t)}
= \dfrac{(t-a) \, \chi(\si(M_{1}/e),t) \, \chi(M_{2},t)}{\chi(M|X,t)}
= (t-a)\chi(\si(M/e),t).
\end{align*}
Thus $ e $ is a divisional atom of $ M $.
\end{proof}

\section{Proof of main theorems}\label{sec: proof}

\subsection{Proof of Theorem \ref{main divisionally free}}
\begin{Lemma}\label{restriction modular flat}
The class $ \mathcal{ME} $ is closed under taking restrictions to modular flats.
\end{Lemma}
\begin{proof}
Let $ M \in \mathcal{ME} $ and $ X $ a modular flat of $ M $.
We proceed by induction on the rank of $ M $.
The case $ r(M)=0 $ is trivial.
Hence we suppose that $ r(M) \geq 1 $.

First assume that $ M $ has a modular coatom $ Z $ such that $ M|Z \in \mathcal{ME} $.
If $ X \subseteq Z $, then $ X $ is a~modular flat of $ M|Z $.
By the induction hypothesis, $ M|X = (M|Z)|X \in \mathcal{ME} $.
Assume $ X \not\subseteq Z $.
Then $ X \vee Z = E $, the ground set of $ M $.
By Proposition~\ref{Brylawski modular intersection}, $ X \cap Z $ is a modular flat of $ M $ and hence $ M|Z $.
By the induction hypothesis, $ M|(X \cap Z) \in \mathcal{ME} $.
Moreover, by the modularity,
\begin{gather*}
r(X) - r(X \cap Z) = r(X \vee Z) - r(Z) = r(E) - r(Z) = 1.
\end{gather*}
Therefore $ X \cap Z $ is a modular coatom of $ M|X $ and hence $ M|X \in \mathcal{ME} $.

Next we suppose that $ M $ is a modular join $ M = P_{Y}(E_{1}, E_{2}) $ over a round flat $ Y $ with $ M|E_{i} \in \mathcal{ME} $ for $ i = 1,2 $.
If $ X \subseteq E_{i} $ for some $ i $, then $ M|X = (M|E_{i})|X \in \mathcal{ME} $ by the induction hypothesis.
Otherwise, $ X_{i} \coloneqq X \cap E_{i} \neq \varnothing $ is a proper subset of $ X $ and $ M|X_{i} = (M|E_{i})|X_{i} \in \mathcal{ME} $ by the induction hypothesis for $ i = 1,2 $.
Since both $ X_{1} $ and $ X_{2} $ are modular by Proposition~\ref{Brylawski modular intersection} and $ X = X_{1} \cup X_{2} $, it follows that $ M|X $ is a modular join of $ M|X_{1} $ and $ M|X_{2} $.
Moreover $ X_{1} \cap X_{2} = X \cap Y $ is round by Proposition~\ref{Probert}.
Therefore $ M|X \in \mathcal{ME} $.
\end{proof}

\begin{Theorem}\label{main matroid}
Every nonempty simple matroid $ M \in \mathcal{ME} $ has a divisional atom $ e $ such that $ \si(M/e) \in \mathcal{ME} $.
\end{Theorem}
\begin{proof}
We will proof the following claims by induction on the rank of $ M $.
\begin{enumerate}[$(i)$]\itemsep=0pt
\item\label{main matroid claim i} If $ M $ has a modular coatom, then there exists a divisional atom $ e $ such that $ \si(M/e) \in \mathcal{ME} $.
\item\label{main matroid claim ii} If $ X $ is a minimal round flat of $ M $ such that $ M $ is a modular join $ M = P_{X}(E_{1}, E_{2}) $.
Then, for each $ i=1,2 $, there exists a divisional atom $ e_{i} \in E_{i}\setminus X $ such that $ \si(M/e_{i}) \in \mathcal{ME} $.
\end{enumerate}

First suppose that $ r(M) = 1 $, that is, the ground set of $ M $ is a singleton.
Then only the case~(\ref{main matroid claim i}) occurs and the atom of $ M $ satisfies the assertion.

Now suppose that $ r(M) \geq 2 $.
If $ M $ has a modular coatom $ X $, then every atom $ e \in E \setminus X $ is divisional and $ \si(M/e) \simeq M|X \in \mathcal{ME} $ by Lemmas~\ref{modular coatom divisional} and~\ref{restriction modular flat}. Thus the assertion holds.

Next we suppose that $ M $ is a modular join.
We assume that $ X $ is a minimal round flat of $ M $ such that $ M = P_{X}(M_{1},M_{2}) $.
Since every modular flat in $ X $~is also round by Proposition~\ref{Probert}, $ X $~is a minimal flat such that $ M $ is a modular join over $ X $.

We will show that $ M_{1} $ has a divisional atom $ e_{1} $ not belonging to $ X $ such that $ \si(M_{1}/e_{1}) \in \mathcal{ME} $.
Note that $ M_{1} $ is a member of $ \mathcal{ME} $ by Lemma~\ref{restriction modular flat}.
Assume that $ M_{1} $ has a modular coatom.
Then, by Proposition~\ref{Ziegler case of modular coatom}, $ M_{1} $ has a divisional atom not belonging to $ X $ such that $ \si(M_{1}/e_{1}) \in \mathcal{ME} $.
Hence we may assume that $ M_{1} $ has a minimal round flat $ Y $ such that $ M_{1} $ is a modular join $ M_{1} = P_{Y}(F,F^{\prime})$.
Since $ X $ is round, we have $ F \supseteq X $ or $ F^{\prime} \supseteq X $.
Without loss of generality, we may assume that $ F^{\prime} \supseteq X $.
By the induction hypothesis, $ M_{1} $ has a divisional atom $ e_{1} \in F\setminus Y $ such that $ \si(M_{1}/e_{1}) \in \mathcal{ME} $.
Assume that $ e_{1} \in X $.
Then $ e_{1} \in F^{\prime} $ and hence $ e_{1} \in F \cap F^{\prime} = Y $, which contradicts $ e_{1} \not\in Y $.
Thus $ M_{1} $ has a divisional atom $ e_{1} $ not belonging to $ X $ such that $ \si(M_{1}/e_{1}) \in \mathcal{ME} $.

By Lemma~\ref{divisional atom of modular join}, $ e_{1} $ is a divisional atom of $ M $.
Moreover, by Proposition \ref{Brylawski contraction}, $ \si(M/e_{1}) $ is isomorphic to a modular join of $ \si(M_{1}/e_{1}) $ and $ M_{2} $, and hence $ \si(M/e_{1}) \in \mathcal{ME} $.
\end{proof}

\begin{proof}[Proof of Theorem~\ref{main divisionally free}] Use Theorem \ref{main matroid} and Proposition \ref{divisional flag}.
\end{proof}

\subsection{Proof of Theorem \ref{main ss}}
\begin{Lemma}\label{lemma ss}
Let $ X $ be a modular coatom of a simple matroid $ M $ on $ E $ such that $ M|X $ is a~modular join $ M|X=P_{Y}(F_{1},F_{2}) $.
Then $ M $ is a modular join $ P_{Y}(F_{1}^{\prime}, F_{2}) $ or $ P_{Y}(F_{1}, F_{2}^{\prime}) $, where~$ F_{i}^{\prime} $ is some flat of $ M $ such that~$ F_{i} $ is a modular coatom of $ M|F_{i}^{\prime} $ for each~$ i $.
\end{Lemma}
\begin{proof}
Recall that, for $ e,e^{\prime} \in E\setminus X $, $ f(e,e^{\prime}) $ denotes a unique element in $ X $ such that $ e,e^{\prime},f(e,e^{\prime}) $ form a circuit (see Proposition~\ref{Jambu-Papadima}).
Assume that there exist three distinct atoms $ a,b,c \in E\setminus X $ such that $ f(a,b) \in F_{1}\setminus F_{2} $ and $ f(a,c) \in F_{2}\setminus F_{1} $.
By Proposition \ref{Jambu-Papadima}, the atoms $ f(a,b),f(a,c),f(b,c) $ form a circuit.
Therefore if $ f(b,c) \in F_{1} $, then $ f(a,c) \in F_{1} $, and if $ f(b,c) \in F_{2} $, then $ f(a,b) \in F_{2} $.
The both cases contradict the assumption.
Hence without loss of generality, we may assume that $ f(a,b) \in F_{1} $ for any two distinct two atoms $ a,b \in E\setminus X $.
Let $ F_{1}^{\prime} \coloneqq (E\setminus X) \cup F_{1} $.

Now we will show that $ M = P_{Y}(F_{1}^{\prime}, F_{2}) $.
Clearly, $ F_{1}^{\prime} \cap F_{2} = F_{1} \cap F_{2} = Y $.
Hence it is satisfied to show that the subset $ F_{1}^{\prime} $ is a modular flat of $ M $.
We will prove it by using Theorem \ref{Brylawsky modular short-circuit axiom}.
Let $ C $ be a circuit of $ M $ and take an atom $ e \in C \setminus F_{1}^{\prime} $.
We will construct a desired circuit by induction on $ m \coloneqq |C \cap (E\setminus X)| $.

First, consider the case $ m=0 $.
By modularity of $ F_{1} $, we have an atom $ x \in F_{1} \subseteq F_{1}^{\prime} $ and a~circuit $ C^{\prime} $ such that $ e \in C^{\prime} \subseteq \{x\} \cup (C\setminus F_{1}) = \{x\} \cup (C\setminus F_{1}^{\prime}) $, which is a desired circuit.

Second, suppose that $ m=1 $.
Let $ C \cap (E\setminus X) = \{a\} $.
Since $ C\setminus \{a\} \subseteq X $ and $ C $ is a circuit, we have $ a \in \cl_{M}(X) = X $, which is a contradiction.
Hence the case $ m=1 $ does not occur.

Finally, assume that $ m \geq 2 $.
Let $ a,b \in C \cap (E\setminus X) $ be distinct atoms.
By Proposition \ref{Jambu-Papadima}, $ T \coloneqq \{a,b,f(a,b)\} $ is a circuit.
Using the strong circuit elimination axiom (see \cite[Proposition~1.4.12]{oxley2011matroid} for example), we obtain a circuit $ C_{1} $ such that $ e \in C_{1} \subseteq (C \cup T)\setminus \{a\} = (C\setminus \{a\}) \cup \{f(a,b)\} $.
Note that $ C_{1}\setminus F_{1}^{\prime} \subseteq C\setminus F_{1}^{\prime} $ since $ f(a,b) \in F_{1} $.
Furthermore, the circuit $ C_{1} $ satisfies $ |C_{1}\cap (E\setminus X)| \leq m-1 $.
Therefore, by the induction hypothesis, we have an atom $ x \in F_{1}^{\prime} $ and a circuit $ C^{\prime} $ such that $ e \in C^{\prime} \subseteq \{x\} \cup (C_{1}\setminus F_{1}^{\prime}) \subseteq \{x\} \cup (C \setminus F_{1}^{\prime}) $, which is a desired circuit.
Thus $ F_{1}^{\prime} $ is a modular flat and hence $ M=P_{Y}(F_{1}^{\prime}, F_{2}) $.
Moreover $ F_{1} $ is a modular flat of $ M|F_{1}^{\prime} $ by Theorem~\ref{Brylawski modular coatom}.
\end{proof}

\begin{proof}[Proof of Theorem~\ref{main ss}]
It suffices to show that every non-supersolvable matroid $ M \in \mathcal{ME} $ is a modular join $ M=P_{Y}(M^{\prime}_{1}, M^{\prime}_{2}) $ over a round flat $ Y $ such that $ M^{\prime}_{1},M^{\prime}_{2} \in \mathcal{ME} $.
We proceed by induction on the rank $ r(M) $.
If $ r(M) \leq 2 $, then $ M $ is supersolvable and we have nothing to prove.
Assume that $ r(M) \geq 3 $.

\looseness=1 We may assume that $ M $ has a modular coatom $ X $ such that $ M|X \in \mathcal{ME} $.
If $ M|X $ is super\-solvable, then so is $ M $, which is a contradiction.
Therefore $ M|X $ is not supersolvable and, by induction, there are simple matroids $ M_{1}, M_{2} \in \mathcal{ME} $ and a round flat $ Y $ such that $ M|X = P_{Y}(M_{1}, M_{2}) $.
By Lemma \ref{lemma ss}, $ M $ is also a modular join $ M = P_{Y}(M^{\prime}_{1}, M^{\prime}_{2}) $ with $ M^{\prime}_{1}, M^{\prime}_{2} \in \mathcal{ME} $.
\end{proof}

\section{Applications}\label{sec:applications}

\subsection{Arrangements associated with gain graphs}
Gain graphs yield two important classes of arrangements.
One includes the Weyl arrangements of type $ A$, $B $, and $ D $ and the other includes the Catalan, Shi, and Linial arrangements.
In this subsection, we study modularly extended matroids associated with gain graphs.

\subsubsection{Basic notions}
A \textit{gain graph} is a tuple $ \Gamma = (V_{\Gamma}, E_{\Gamma}, L_{\Gamma}, G_{\Gamma}) $, where
\begin{itemize}\itemsep=0pt
\item $ V_{\Gamma} $ is a finite set,
\item $ L_{\Gamma} $ is a subset of $ V_{\Gamma} $,
\item $ G_{\Gamma} $ is a group,
\item $ E_{\Gamma} $ is a finite subset of $ \{ (u,v,g) \in V_{\Gamma} \times V_{\Gamma} \times G_{\Gamma} \,|\, u \neq v \} $ divided by the equivalence relation~$ \sim $ generated by $ (u,v, g) \sim \big(v,u, g^{-1}\big) $.
\end{itemize}
Let $ \{u,v\}_{g} $ denote the equivalence class containing $ (u,v,g) $ and hence $ \{u,v\}_{g} = \{v,u\}_{g^{-1}} $.
Elements of $ V_{\Gamma}$, $E_{\Gamma}$, and $ L_{\Gamma} $ are called \textit{vertices}, \textit{edges}, and \textit{loops} of the gain graph $ \Gamma $ respectively and~$ G_{\Gamma} $ is called the \textit{gain group} of~$ \Gamma $.
We quite simplify the notion of gain graphs. See Zaslavsky~\cite{zaslavsky1989biased-joctsb} for a general treatment.
Note that every simple graph can be regarded as a loopless gain graph over the trivial group.

A \textit{cycle} of a gain graph $ \Gamma $ a loop or a subset of $ E_{\Gamma} $ consisting of edges
\begin{gather*}
\{v_{1}, v_{2} \}_{g_{1}}, \{v_{2}, v_{3} \}_{g_{2}}, \dots, \{v_{m-1}, v_{m} \}_{g_{m-1}}, \{v_{m}, v_{1} \}_{g_{m}}
\end{gather*}
with distinct vertices $ v_{1}, \dots, v_{m}$ $(m \geq 2) $, where $ \{v_{1}, v_{2} \}_{g_{1}} \neq \{v_{2}, v_{1} \}_{g_{2}} $ if $ m=2 $. The cycle above is said to be \textit{balanced} if $ g_{1}g_{2} \cdots g_{m} = 1 $.
Note that whether or not the value equals the identity is independent of indexing the vertices of the cycle and hence being balanced is well-defined. Every loop is defined to be unbalanced.

A subset $ S \subseteq E_{\Gamma} \sqcup L_{\Gamma} $ is called \textit{balanced} if every cycle in $ S $ is balanced (and hence $ S $ has no loops). The set $ S $ is said to be \textit{contrabalanced} if $ S $ has no balanced cycles.
Moreover, $ S $ is called \textit{balance-closed} if
\begin{align*}
\{ e \in E_{\Gamma} \setminus S \,|\, \text{ there exists a balanced cycle } C \text{ such that } e \in C \subseteq S \cup \{e\}\} = \varnothing.
\end{align*}

A \textit{path} on distinct vertices $ v_{1}, \dots, v_{m}$ $(m \geq 1) $ is a subset of $ E_{\Gamma} $ consisting of edges
\begin{gather*}
\{v_{1}, v_{2} \}_{g_{1}}, \{v_{2}, v_{3} \}_{g_{2}}, \dots, \{v_{m-1}, v_{m} \}_{g_{m-1}}.
\end{gather*}
A \textit{tight handcuff} is the union of two cycles $ C_{1} $ and $ C_{2} $ such that $ C_{1} $ and $ C_{2} $ have just one common vertex.
A \textit{loose handcuff} is the union of two cycles $ C_{1} $ and $ C_{2} $, and a path $ P $ from $ v_{1} $ to $ v_{2} $ of positive length such that $ P $ and $ C_{i} $ meet only at $ v_{i} $ and the cycles $ C_{1} $ and $ C_{2} $ does not share vertices.
A \textit{handcuff} is a tight or loose handcuff.
A \textit{theta} is the union of three paths meeting only at their endvertices.

Suppose that $ G $ is a finite group.
Let $ K_{n}^{G} $ denote the loopless gain graph on the vertex set $ [n] = \{1, \dots, n\} $ with gain group $ G $ and edges
\begin{align*}
\{ \{i,j\}_{g} \,|\, i,j \in [n] \text{ with } i \neq j \text{ and } g \in G\}
\end{align*}
and let $ \mathring{K}_{n}^{G} $ denote the gain graph $ K_{n}^{G} $ together with all possible loops.
Note that both of $ K_{0}^{G} $ and $ \mathring{K}_{0}^{G} $ mean the null graph.

A gain graph is \textit{connected} if there exists a path between every pair of vertices of the graph.
If a gain graph is disconnected, then it is decomposed into the connected components in a usual manner.
A connected component of a subset $ S \subseteq E_{\Gamma} \sqcup L_{\Gamma} $ is a connected component of the gain graph $ (V_{\Gamma}, S\cap E_{\Gamma}, S\cap L_{\Gamma}, G_{\Gamma}) $.

Let $ W \subseteq V_{\Gamma} $.
A \textit{subgraph induced by $ W $} is a gain graph $ \Gamma[W] = (W, E_{\Gamma[W]}, W \cap L_{\Gamma}, G_{\Gamma}) $, where $ E_{\Gamma[W]} \coloneqq \{ \{u,v\}_{g} \in E_{\Gamma} \,|\, u,v \in W \} $.
An \textit{induced subgraph} of $ \Gamma $ is a subgraph induced by some subset of $ V_{\Gamma} $.
Moreover, $ \Gamma \setminus v \coloneqq \Gamma[V_{\Gamma}\setminus\{v\}] $.

\subsubsection{Frame matroids and the associated arrangements}
\begin{Theorem}[{Zaslavsky \cite[Theorem 2.1]{zaslavsky1991biased-joctsb}}]\label{Zaslavsky frame matroid}
Let $ \Gamma $ be a gain graph.
Then the following conditions define the same matroid on $ E_{\Gamma} \sqcup L_{\Gamma} $.
\begin{enumerate}[$(a)$]\itemsep=0pt
\item\label{Zaslavsky frame matroid a} A subset of $ E_{\Gamma} \sqcup L_{\Gamma} $ is independent if and only if every connected component of it has no balanced cycles and at most one unbalanced cycle.
\item\label{Zaslavsky frame matroid b} A subset of $ E_{\Gamma} \sqcup L_{\Gamma} $ is a circuit if and only if it is a balanced cycle, a contrabalanced handcuff, or a contrabalanced theta.
\end{enumerate}
We call the matroid the \textit{frame matroid} of $ \Gamma $, denoted by $ M_{\times}(\Gamma) $.
\end{Theorem}

The frame matroid $ M_{\times}\big(\mathring{K}_{n}^{G}\big) $ is known as the Dow\-ling geometry~$ Q_{n}(G) $ introduced by Dow\-ling~\cite{dowling1973class-joctsb}.
When $ n \geq 3 $, the Dowling geometry $ Q_{n}(G) $ is representable over a field $ \mathbb{K} $ if and only if~$ G $ is isomorphic to a subgroup of $ \mathbb{K}^{\times} $ \cite[Theorem~9]{dowling1973class-joctsb}.

When $ \Gamma $ is a gain graph on $ [n] $ whose gain group $ G $ is a subgroup of the multiplicative group~$ \mathbb{K}^{\times} $ of a field $ \mathbb{K} $, we may associate~$ \Gamma $ with an arrangement $ \mathcal{A}_{\times}(\Gamma) $ in $ \mathbb{K}^{n} $ defined by the following
\begin{gather*}
\mathcal{A}_{\times}(\Gamma) \coloneqq \{\{x_{i}-gx_{j} = 0\} \,|\, \{i,j\}_{g} \in E_{\Gamma}\} \cup \{\{x_{i}=0\} \,|\, i \in L_{\Gamma}\}.
\end{gather*}

\begin{Example}
When $ G = \{1\} \subseteq \mathbb{K}^{\times} $, the arrangement $ \mathcal{A}_{\times}(\Gamma) $ is the graphic arrangement.
Especially, $ \mathcal{A}_{\times}\big(K_{n}^{\{1\}}\big) $ is the braid arrangement, also known as the Weyl arrangement of type~$ A_{n-1} $.
If $ G = \{\pm 1\} \subseteq \mathbb{R}^{\times} $, then the arrangements $ \mathcal{A}_{\times}\big(\mathring{K}_{n}^{\{\pm 1\}}\big) $ and $ \mathcal{A}_{\times}\big(K_{n}^{\{\pm 1\}}\big) $ are known as the Weyl arrangements of type $ B_{n} $ and $ D_{n} $.
More specifically,
\begin{gather*}
\mathcal{A}_{\times}\big(K_{n}^{\{1\}}\big) = \{\{x_{i}-x_{j}=0\} \,|\, 1 \leq i < j \leq n\}, \\
\mathcal{A}_{\times}\big(\mathring{K}_{n}^{\{\pm 1\}}\big) = \{\{x_{i} \pm x_{j}=0\} \,|\, 1 \leq i < j \leq n\} \cup \{\{x_{i}=0\} \,|\, 1 \leq i \leq n\}, \\
\mathcal{A}_{\times}\big(K_{n}^{\{\pm 1\}}\big) = \{\{x_{i} \pm x_{j}=0\} \,|\, 1 \leq i < j \leq n\}.
\end{gather*}
\end{Example}

\begin{Theorem}[{Zaslavsky \cite[Theorem~2.1(a)]{zaslavsky2003biased-joctsb}}]\label{Zaslavsky frame arrangement}
The linear dependence matroid on $ \mathcal{A}_{\times}(\Gamma) $ is isomorphic to the frame matroid $ M_{\times}(\Gamma) $.
\end{Theorem}

Note that if $ \Gamma $ is a simple graph, then $ M_{\times}(\Gamma) $ and $ \mathcal{A}_{\times}(\Gamma) $ coincide with the graphic matroid and arrangement. Moreover, recall that the graphic matroids associated with complete graphs are round.
Here we have a generalization for frame matroids.

\begin{Proposition}\label{frame round}
Suppose that $ G $ is a finite group.
Then the frame matroids $ M_{\times}\big(K_{n}^{G}\big) $ and $ M_{\times}\big(\mathring{K}_{n}^{G}\big) $ are round except for $ M_{\times}\big(K_{2}^{\{\pm 1\}}\big) $.
\end{Proposition}
\begin{proof}
First we prove that $ M_{\times}\big(K_{n}^{G}\big) $ is round except for $ M_{\times}\big(K_{2}^{\{\pm 1\}}\big) $.
Let $ E $ be the ground set of $ M_{\times}\big(K_{n}^{G}\big) $ and suppose that $ E = F_{1} \cup F_{2} $.

Now consider the case $ n = 2 $.
If $ |G| = 1 $, then $ M_{\times}\big(K_{2}^{G}\big) $ is trivially round.
Assume that $ |G|\geq 3 $.
Then we may suppose that $ |F_{1}| \geq 2 $.
Since every contrabalanced handcuff is a circuit by Theorem~\ref{Zaslavsky frame matroid}(\ref{Zaslavsky frame matroid b}), we have $ F_{1} = E $ and hence $ M_{\times}\big(K_{2}^{G}\big) $ is round.

Suppose that $ n \geq 3 $.
Define a relation $ \sim $ on the vertex set $ K_{n}^{G} $ by $ u \sim v $ if $ u=v $ or there exists an edge in $ F_{1} $ connecting $ u $ and $ v $.
Since every balanced triangle is a circuit by Theorem~\ref{Zaslavsky frame matroid}(\ref{Zaslavsky frame matroid b}), the relation $ \sim $ is an equivalence relation.
Assume that there exist two or more distinct equivalence classes.
Consider the equivalence relation~$ \approx $ among vertices defined by $ u \approx v $ if $ u=v $ or there exists an edge in~$ F_{2} $ between~$ u $ and~$ v $.
If $ u \not\sim v $, then $ u \approx v $ since $ E=F_{1} \cup F_{2} $.
When $ u \sim v $, take a vertex $ w $ such that $ u \not\sim w $ and $ v \not\sim w $.
Then $ u \approx w $ and $ v \approx w $.
Therefore $ u \approx v $.
Namely, all vertices are equivalent under the relation~$ \approx $.
Hence without loss of generality we may assume that every two distinct vertices are connected by an edge in~$ F_{1} $.

If $ |G| = 1 $, then $ F_{1} = E $ and hence $ M_{\times}\big(K_{2}^{G}\big) $ is round.
Therefore we assume $ |G| \geq 2 $.
Suppose there exists a pair of vertices such that there exist two different edges in $ F_{1} $ connecting them.
Since a contrabalanced theta and a contrabalanced tight handcuff are circuits by Theorem~\ref{Zaslavsky frame matroid}(\ref{Zaslavsky frame matroid b}), we have $ F_{1} = E $ and $ M_{\times}\big(K_{2}^{G}\big) $ is round.
Hence we may assume that there exist no such pairs, that is, there exists exactly one edge in~$ F_{1} $ between each pair of vertices.

If $ |G| \geq 3 $, then $ F_{2} $ has at least two edges between every pair of vertices, which implies $ F_{2} = E $.
Therefore $ M_{\times}\big(K_{2}^{G}\big) $ is round.
Hence we may assume that $ |G|=2 $.
Focus on a triangle in~$ F_{1} $.
If the triangle is unbalanced, then there exists another edge in $ F_{1} $ such that it forms a balanced triangle with two edges in the triangle since every balanced triangle is a circuit.
Therefore~$ F_{1} $ has a pair of vertices such that there exist at least two edges between them, which contradicts to the assumption of~$ F_{1} $.
Therefore the triangle is balanced and hence~$ F_{2} $ has an unbalanced triangle.
By the same argument, $ F_{2} $ has a pair of vertices such that there exist at least two edges between them.
This implies that $ F_{2}=E $.
Thus $ M_{\times}\big(K_{2}^{G}\big) $ is round.

Next, we prove that $ M_{\times}\big(\mathring{K}_{n}^{G}\big) $ is round.
Let $ S $ be the subset of the ground set $ E $ of $ M \coloneqq M_{\times}\big(\mathring{K}_{n}^{G}\big) $ corresponding to the subgraph $ \mathring{K}_{n}^{\{1\}} $.
Then $ M|S = M_{\times}\big(\mathring{K}_{n}^{\{1\}}\big) \simeq M(K_{n+1}) $, which is round.
Every edge $ \{i,j\}_{g} $ of $ \mathring{K}_{n}^{G} $ forms a contrabalanced handcuff with the loops attached to the endvertices $ i $ and~$ j $.
Since a contrabalanced handcuff is a circuit by Theorem~\ref{Zaslavsky frame matroid}(\ref{Zaslavsky frame matroid b}), we have $ \cl_{M}(S) = E $. The assertion holds by the following proposition.
\end{proof}

\begin{Proposition}[{Kung \cite[Lemma 4.1]{kung1986numerically-gd}}, {Probert \cite[Lemma 4.2.7]{probert2018chordality}}]
Let $ S $ be a subset of the ground set of a matroid $ M $.
If $ S $ is round, then $ \cl_{M}(S) $ is round.
\end{Proposition}

\begin{Example}
The matroids on Weyl arrangements $ \mathcal{A}_{\times}\big(K_{n}^{\{1\}}\big) $ of type $ A_{n-1} $ and $ \mathcal{A}_{\times}\big(\mathring{K}_{n}^{\{\pm 1\}}\big) $ of type $ B_{n} $ are round.
The matroid on Weyl arrangement $ \mathcal{A}_{\times}\big(K_{n}^{\{\pm 1\}}\big) $ of type $ D_{n} \ (n \geq 3) $ is also round.
\end{Example}

In the case of simple graphs, recall that a subgraph isomorphic to a complete graph corresponds to a modular flat.
Here is a generalization for frame matroids.

\begin{Proposition}\label{frame modular}
Let $ \Gamma $ be a gain graph with a finite gain group $ G $.
Suppose that $ \Gamma $ has an induced subgraph isomorphic to $ \mathring{K}_{n}^{G} $.
Then the corresponding flat of~$ M_{\times}(\Gamma) $ is modular.
\end{Proposition}
\begin{proof}
Let $ X $ denote the corresponding flat.
We will prove modularity of $ X $ by using Theorem~\ref{Brylawsky modular short-circuit axiom}.
Let $ C $ be a circuit and take an atom $ e \in C \setminus X $.
Suppose that $ S $ is the connected component of $ C \setminus X $ containing $ e $.
We may assume that $ C \cap X \neq \varnothing $.
Then $ S $ is an independent set.
By Theorem~\ref{Zaslavsky frame matroid}(\ref{Zaslavsky frame matroid a}), $ S $ has no balanced cycles and at most one unbalanced cycle. Moreover~$ S $ has at least one vertex of the subgraph~$ \mathring{K}_{n}^{G} $ since $ C \cap X \neq \varnothing $.

First, assume that $ S $ has an unbalanced cycle containing $ e $ (including the case $ e $ itself is a loop).
Then the unbalanced cycle and the loop $ x \in X $ of a vertex belonging to both $ S $ and $ \mathring{K}_{n}^{G} $ with the path connecting them form a handcuff~$ C^{\prime} $, which is a desired circuit since $ e \in C^{\prime} \subseteq \{x\} \cup S \subseteq \{x\} \cup (C\setminus X) $.

Second, suppose that $ S $ has an unbalanced cycle not containing $ e $.
If we delete the unbalanced cycle, then the remaining graph is a forest.
Hence we can obtain a path which contains $ e $ and connects the unbalanced cycle and a leaf.
The leaf is a vertex of the subgraph $ \mathring{K}_{n}^{G} $ since every circuit has no leaves by Theorem \ref{Zaslavsky frame matroid}(\ref{Zaslavsky frame matroid b}).
Then the loop $ x \in X $ of the leaf and the unbalanced cycle with the path form a handcuff, which is a desired circuit.

Finally, consider the case $ S $ has no unbalanced cycle.
In this case, $ S $ is a tree.
Therefore we can obtain a path which contains $ e $ and connecting leaves of $ S $.
This path contains at least $ 3 $ vertices since $ e \neq X $.
The endvertices of the path belong to the subgraph $ \mathring{K}_{n}^{G} $.
We can choose $ x \in X $ between the endvertices such that~$ x $ and the path form a balanced cycle, which is a~desired circuit and hence we can conclude~$ X $ is modular.
\end{proof}

\begin{Remark}
Contrary to Proposition \ref{frame modular}, $ K_{n}^{G}$ $(|G| \geq 2) $ may yield a non-modular flat.
For example, see Fig.~\ref{Fig:fish}.
The edges between the middle and right vertices denote the flat corresponding to $ K_{2}^{G} $.
Choose two edges in it and consider the contrabalanced loose handcuff formed with the loop and the edge between the left and middle vertices, which is a circuit by Theorem~\ref{Zaslavsky frame matroid}(\ref{Zaslavsky frame matroid b}).
Using Theorem~\ref{Brylawsky modular short-circuit axiom}, we can conclude that the flat corresponding to $ K_{2}^{G} $ is not modular.
In a~similar way, we can construct a frame matroid in which the flat corresponding to $ K_{n}^{G}$ $(n \geq 3) $ is not modular.
\end{Remark}
\begin{figure}[t]\centering

\begin{tikzpicture}
\draw (0,0) node[v](1){};
\draw (2,0) node[v](2){};
\draw (4,0) node[v](3){};
\draw (1)--(2)--(3);
\draw[scale=4] (1) to[in=135,out=225,loop] (1);
\draw (2) to[bend left=90] (3);
\draw (2) to[bend left=45] (3);
\draw (2) to[bend left=20] (3);
\draw (2) to[bend right=90] (3);
\draw (2) to[bend right=45] (3);
\draw (2) to[bend right=20] (3);
\end{tikzpicture}
\caption{The flat corresponding to $ K_{2}^{G} $ is not modular.}\label{Fig:fish}
\end{figure}
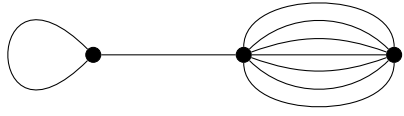

Next, we introduce bias-simplicial vertices, which is a generalization of simplicial vertices of simple graphs.
\begin{Definition}
A vertex $ v $ in a gain graph $ \Gamma $ is called \textit{bias simplicial} if the following conditions hold.
\begin{enumerate}[(i)]\itemsep=0pt
\item If $ \{u,v\}_{g}, \{v,w\}_{h} \in E_{\Gamma} $, then $ \{u,w\}_{gh} \in E_{\Gamma} $.
\item If $ \{u,v\}_{g}, \{u,v\}_{h} \in E_{\Gamma} $, and $ g \neq h $, then $ u \in L_{\Gamma} $.
\item If $ \{u,v\}_{g} \in E_{\Gamma} $ and $ v \in L_{\Gamma} $, then $ u \in L_{\Gamma} $.
\end{enumerate}
\end{Definition}

Zaslavsky \cite[Theorem 2.1]{zaslavsky2001supersolvable-ejoc} characterized modular coatoms of frame matroids.
The following theorem is an excerpt.
(Note that one type of modular coatom is missing in the classification. See Koban \cite[Theorem $ 2.1^{\prime} $]{koban2004comments-ejoc} for the complete classification.)
\begin{Theorem}[{Zaslavsky \cite[Theorem 2.1(1)]{zaslavsky2001supersolvable-ejoc}}]\label{Zaslavsky bias simplicial}
Let $ \Gamma $ be a gain graph and $ v $ a bias-simplicial vertex.
Then the flat of $ M_{\times}(\Gamma) $ corresponding to the induced subgraph $ \Gamma \setminus v $ is modular.
\end{Theorem}

One can show that the frame matroid $ M_{\times}\big(\mathring{K}_{n}^{G}\big) $ is supersolvable for any finite group $ G $ by Theorem~\ref{Zaslavsky bias simplicial}.

Zaslavsky \cite[Theorem 2.2]{zaslavsky2001supersolvable-ejoc} characterized supersolvability of frame matroids as the minimal class of gain graphs which satisfies the conditions (i)--(v) in the following theorem and is closed under taking disjoint unions.
We can replace disjoint unions with modular joins for modular extendedness (condition (vi)).

\begin{Theorem}\label{main frame}
Let $ G $ be a finite subgroup of the multiplicative group of a field $ \mathbb{K} $ and $ \mathcal{C}^{G}_{\times} $ the minimal class of gain graphs with gain group $ G $ which satisfies the following conditions.
\begin{enumerate}[$(i)$]\itemsep=0pt
\item The null graph is a member of $ \mathcal{C}^{G}_{\times} $.
\item $ K_{2}^{G} \in \mathcal{C}^{G}_{\times} $.
\item If $ \{\pm 1 \} \subseteq G $, then $ K_{3}^{\{\pm 1\}} \in \mathcal{C}^{G}_{\times} $.
\item If $ \{\pm 1 \} \subseteq G $, then every connected loopless gain graph $ \Gamma $ over $ \{\pm 1\} $ such that the positive edges form a chordal graph, the negative edges form a star $ \{u, v_{1} \}_{-1}, \dots, \{u, v_{r}\}_{-1} $, and $ v_{1}, \dots, v_{r} $ form a clique consisting of positive edges is a member of $ \mathcal{C}^{G}_{\times} $.
\item If $ \Gamma $ has a bias-simplicial vertex $ v $ and $ \Gamma \setminus v \in \mathcal{C}^{G}_{\times}$.
\item If there exists a decomposition $ V_{\Gamma} = V_{1} \cup V_{2} $ such that $ \Gamma[V_{1}], \Gamma[V_{2}] \in \mathcal{C}^{G}_{\times}, \ E_{\Gamma} = E_{\Gamma[V_{1}]} \cup E_{\Gamma[V_{2}]} $, and $ \Gamma[V_{1} \cap V_{2}] \simeq \mathring{K}_{n}^{G} $ for some $ n \geq 0 $, then $ \Gamma \in \mathcal{C}^{G}_{\times} $.
\end{enumerate}
Then for every $ \Gamma \in \mathcal{C}^{G}_{\times} $ the corresponding arrangement $ \mathcal{A}_{\times}(\Gamma) $ is divisionally free.
\end{Theorem}
\begin{proof}
By Propositions \ref{frame round}, \ref{frame modular}, and Theorem \ref{Zaslavsky bias simplicial}, the frame matroid $ M_{\times}(\Gamma) $ is modularly extended. By Theorems~\ref{main divisionally free} and~\ref{Zaslavsky frame arrangement} we can conclude that $ \mathcal{A}_{\times}(\Gamma) $ is divisionally free.
\end{proof}

\begin{Example}\label{example frame}
Let $ \bowtie $ be a signed graph described in Fig.~\ref{Fig:bowtie}, where a \textit{signed graph} is a gain graph with gain group $ \{\pm 1\} $.
Let $ \mathring{\bowtie} $ denote the signed graph $ \bowtie $ with the loops attached to every vertex.
Then $ \mathring{\bowtie} \in \mathcal{C}^{\{\pm 1\}}_{\times} $ and the arrangement $ \mathcal{A}_{\times}(\mathring{\bowtie}) $ is the arrangement $ \mathcal{A} $ in Example \ref{example intro}, that is, an arrangement consisting of the following hyperplanes
\begin{gather*}
\{z=0\},\,\{x_{1}=0\},\,\{x_{2}=0\},\,\{x_{1}-z=0\},\,\{x_{2}-z=0\},\,\{x_{1}-x_{2}=0\},\,\{x_{1}+x_{2}=0\}, \\
\{y_{1}=0\},\,\{y_{2}=0\},\,\{y_{1}-z=0\},\,\{y_{2}-z=0\},\,\{y_{1}-y_{2}=0\},\,\{y_{1}+y_{2}=0\}.
\end{gather*}
The signed graph $ \mathring{\bowtie} $ is not of type (i)--(iv) in Theorem~\ref{main frame}.
Moreover, $ \mathring{\bowtie} $ has no bias-simplicial vertex. Therefore $ \mathcal{A}_{\times}(\mathring{\bowtie}) $ is not supersolvable but divisionally free.
\end{Example}

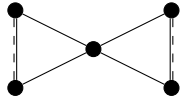
\begin{figure}[t]\centering
\begin{tikzpicture}[scale=.6]
\draw (0,0) node[v](z){};
\draw (-1.73, 0.86) node[v](x1){};
\draw (-1.73, -0.86) node[v](x2){};
\draw ( 1.73, 0.86) node[v](y1){};
\draw ( 1.73, -0.86) node[v](y2){};
\draw[decoration={dashsoliddouble}, decorate] (x2)--(x1);
\draw[decoration={dashsoliddouble}, decorate] (y1)--(y2);
\draw (x1)--(z)--(x2);
\draw (y1)--(z)--(y2);
\end{tikzpicture}
\caption{The signed graph $ \bowtie $ (Dashed line segments denote negative edges).}\label{Fig:bowtie}
\end{figure}

\begin{Remark}
Since $ K_{n}^{\{1\}}, \mathring{K}_{n}^{\{\pm 1\}} \in \mathcal{C}_{\times}^{\{\pm 1\}} $, the Weyl arrangements of type $ A_{n-1} $ and $ B_{n} $ are divisionally free (actually these are supersolvable).

If $ n \geq 4 $, then $ K_{n}^{\{\pm 1\}} \not\in \mathcal{C}_{\times}^{\{\pm 1\}} $.
Actually the frame matroid $ M_{\times}\big(K_{n}^{\{\pm 1\}}\big) $ is not modularly extended since it is round by Proposition~\ref{frame round} and has no modular coatoms by \cite[Theorem~2.1]{zaslavsky2001supersolvable-ejoc} and \cite[Theorem~$2.1^{\prime} $]{koban2004comments-ejoc}.

However, it is well known that every Weyl arrangement is free by Saito \cite{saito1977uniformization-rk, saito1980theory-jotfostuotsam}, including the Weyl arrangement $ \mathcal{A}_{\times}\big(K_{n}^{\{\pm 1\}}\big) $ of type $ D_{n} $, which is also inductively free (see \cite[Example~2.6]{jambu1984free-aim}).
Thus the Weyl arrangement of type $ D_{n}$ $(n \geq 4) $ is inductively free (and hence divisionally free) but not modularly extended.
\end{Remark}

\begin{Question}Does there exist a modularly extended arrangement which is not inductively free?
\end{Question}

\subsubsection{Extended lift matroids and associated arrangements}
\begin{Theorem}[{Zaslavsky \cite[Theorem~3.1]{zaslavsky1991biased-joctsb}}]\label{Zaslavsky extended lift matroid}
Let $ \Gamma $ be a loopless gain graph.
Then the following conditions define the same matroid on $ E_{\Gamma} \sqcup \{\infty\} $.
\begin{enumerate}[$(a)$]\itemsep=0pt
\item\label{Zaslavsky extended lift matroid a} A subset of $ E_{\Gamma} \sqcup \{\infty\} $ is independent if and only if has no balanced cycle and contains at most either $ \infty $ or one unbalanced cycle.
\item\label{Zaslavsky extended lift matroid b} A subset of $ E_{\Gamma} \sqcup \{\infty\} $ is a circuit if and only if it is a balanced cycle, a contrabalanced tight handcuff, a contrabalanced theta, the union of two vertex-disjoint unbalanced cycles, or the union of $ \{\infty\} $ and an unbalanced cycle.
\item\label{Zaslavsky extended lift matroid c} A subset $ X \in E_{\Gamma} \sqcup \{\infty\} $ is a flat if and only if $ X $ satisfies the one of the following conditions.
\begin{enumerate}[$(i)$]\itemsep=0pt
\item$ X \not\ni \infty $ and $ X $ is balanced and balance-closed.
\item $ X \ni \infty $ and $ X\setminus\{\infty\} $ is the union of the edge sets of the induced subgraphs $ \Gamma[W_{1}], \dots,\allowbreak \Gamma[W_{r}] $, where $ W_{1}, \dots, W_{r} $ are mutually disjoint subsets of $ V_{\Gamma} $.
\end{enumerate}
\end{enumerate}
We call the matroid the \textit{extended lift matroid} of $ \Gamma $, denoted $ M_{+}(\Gamma) $.
\end{Theorem}

When $ \Gamma $ is a loopless gain graph on $ [n] $ whose gain group is a subgroup of the additive group~$ \mathbb{K}^{+} $ of a field $ \mathbb{K} $, we may associate $ \Gamma $ with an arrangement~$ \mathcal{A}_{+}(\Gamma) $ in~$ \mathbb{K}^{n+1} $ defined by the following
\begin{gather*}
\mathcal{A}_{+}(\Gamma) \coloneqq \{\{z=0\}\} \cup \{\{x_{i}-x_{j}=gz\} \,|\, \{i,j\}_{g} \in E_{\Gamma}\},
\end{gather*}
where $ z, x_{1}, \dots, x_{n} $ denote the coordinate of~$ \mathbb{K}^{n+1} $.
This arrangement is the cone over the affine arrangement consisting of hyperplanes corresponding to edges of~$ \Gamma $ and the element $ \infty $ corresponds to the hyperplane at infinity~$ \{z=0\} $.

\begin{Remark}
Consider gain graphs with gain group $ \mathbb{Z} $.
For a positive integer $ a $, the arrangements $ \mathcal{A}_{+}(\Gamma) $ with edge sets
\begin{gather*}
 \{\{i,j\}_{g} \,|\, 1 \leq i < j \leq n, \, g \in \{-a,-a+1 \dots, a\}\}, \\
 \{\{i,j\}_{g} \,|\, 1 \leq i < j \leq n, \, g \in \{-a+1, -a+2, \dots, a\}\}, \\
 \{\{i,j\}_{g} \,|\, 1 \leq i < j \leq n, \, g \in \{-a+2,-a+3, \dots, a\}\},
\end{gather*}
are the cones of the extended Catalan, Shi, and Linial arrangements.
The cones of the extended Catalan and Shi arrangements are known to be free \cite{athanasiadis1998free-ejoc,edelman1996free-dcg, yoshinaga2004characterization-im}.

Postnikov and Stanley \cite{postnikov2000deformations-joctsa} computed the characteristic polynomials of the deformations of Weyl arrangement of type~$ A $, including these arrangements.
Moreover, all roots of the characteristic polynomial of the extended Linial arrangement have real part $ (2a-1)n/2 $ \cite[Theorem~9.12]{postnikov2000deformations-joctsa}.
Combining Terao's factorization theorem \cite{terao1981generalized-im}, we have that if $ n = 3 $, then the cone over the extended Linial arrangement is not free and thus every cone over the extended Linial arrangement is not free since it contains the extended Linial arrangement of dimension~$ 3 $ as a~localization (see, for example \cite[Theorem~4.37]{orlik1992arrangements}).

Recently, Nakashima and the author \cite{nakashima2019enumeration-a} give explicit formulas for the number of flats of extended Catalan and Shi arrangements with theory of gain graphs and combinatorial species.
\end{Remark}

\begin{Theorem}[{Zaslavsky \cite[Theorem 3.1(a)]{zaslavsky2003biased-joctsb}}]\label{Zaslavsky extended lift arrangement}
The linear dependence matroid on $ \mathcal{A}_{+}(\Gamma) $ is isomorphic to the extended lift matroid $ M_{+}(\Gamma) $.
\end{Theorem}

Note that if $ \Gamma $ is a simple graph, then $ M_{+}(\Gamma) $ and $ \mathcal{A}_{+}(\Gamma) $ are the graphic matroid and arrangement with an extra element independent from the other elements.
Recall again that a subgraph isomorphic to a complete graph yields a round and modular flat.
The following propositions are generalizations for extended lift matroids.

\begin{Proposition}\label{extended lift round}
Suppose that $ G $ is a non-trivial finite group.
Then the extended lift matroid $ M_{+}\big(K_{n}^{G}\big) $ is round.
\end{Proposition}
\begin{proof}
When $ n=1 $, the assertion is trivial since $ M_{+}\big(K_{1}^{G}\big) = \{\infty\} $.
Hence we suppose that $ n \geq 2 $ and assume that there exist two proper flats $ X $ and $ Y $ such that $ X \cup Y = E_{\Gamma} \sqcup \{\infty\} $.
We may assume that $ X \ni \infty $.
Then there exist mutually disjoint subset $ W_{1}, \dots, W_{r} $ of $ V_{\Gamma} $ such that $ X \setminus \{\infty\} $ is the union of the edge sets of the induced subgraphs $ \Gamma[W_{1}], \dots, \Gamma[W_{r}] $ by Theorem~\ref{Zaslavsky extended lift matroid}(\ref{Zaslavsky extended lift matroid c}).
Note that $ r \geq 2 $ since $ X $ is a proper flat.

If $ Y \not\ni \infty $, then $ Y $ is balanced by Theorem~\ref{Zaslavsky extended lift matroid}(\ref{Zaslavsky extended lift matroid c}).
However, $ Y $ must contain all edges between a vertex in $ W_{1} $ and a vertex in~$ W_{2} $, and hence~$ Y $ is unbalanced since the gain group $ G $ is non-trivial, which is a contradiction.

Now suppose that $ Y \ni \infty $.
Let $ Z $ be the flat consisting of the edges of the subgraph $ K_{n}^{\{1\}} $.
Then we have $ Z = (X\cap Z) \cup (Y\cap Z) $.
This contradicts the roundness of the graphic matroid of the complete graph.
Hence we can conclude that the assertion holds.
\end{proof}

\begin{Proposition}\label{extended lift modular}
Let $ \Gamma $ be a loopless gain graph with a finite gain group $ G $.
Suppose that $ G $ has an induced subgraph isomorphic to $ K_{n}^{G} $.
Then the corresponding flat of $ M_{+}(\Gamma) $ is modular.
\end{Proposition}
\begin{proof}
Let $ X $ denote the corresponding flat and we show $ X $ is modular by using Theorem~\ref{Brylawsky modular short-circuit axiom}.
Let $ C $ be a circuit and $ e \in C \setminus X $.
We may assume that $ C \cap X \neq \varnothing $.
Then $ C\setminus X \not\ni \infty $ and it is independent has no balanced cycle and contains at most one unbalanced cycle by Theorem~\ref{Zaslavsky extended lift matroid}(\ref{Zaslavsky extended lift matroid a}).

Assume that $ C\setminus X $ has an unbalanced cycle containing $ e $.
Then the unbalanced cycle and $ \infty $ form a circuit by Theorem~\ref{Zaslavsky extended lift matroid}(\ref{Zaslavsky extended lift matroid b}), which is a desired circuit.

Now suppose that there exists no unbalanced cycle containing $ e $.
By Theorem~\ref{Zaslavsky extended lift matroid}(\ref{Zaslavsky extended lift matroid b}), there exists a cycle in $ C $ containing $ e $.
Therefore we can find a path in $ C\setminus X $ containing $ e $ whose endvertices belong to the subgraph~$ K_{n}^{G} $
Choose a suitable edge between the endvertices, and we obtain a balanced cycle, which is a desired circuit.
Thus we can conclude that the assertion holds true.
\end{proof}

We introduce link-simplicial vertices which are another generalization of simplicial vertices of simple graphs and are fit to loopless gain graphs and extended lift matroids.
\begin{Definition}
A vertex $ v $ in a loopless gain graph $ \Gamma $ is called \textit{link simplicial} if the following condition holds:
\begin{itemize}\itemsep=0pt
\item If $ \{u,v\}_{g}, \{v,w\}_{h} \in E_{\Gamma} $, then $ \{u,w\}_{gh} \in E_{\Gamma} $.
\end{itemize}
\end{Definition}

Zaslavsky characterized modular coatoms of extended lift matroids.
We excerpt from the theorem.
\begin{Theorem}[{Zaslavsky \cite[Theorem~3.1]{zaslavsky2001supersolvable-ejoc}}]\label{Zaslavsky link simplicial}
Let $ \Gamma $ be a loopless gain graph and~$ v $ a link-simplicial vertex.
Then the flat of $ M_{+}(\Gamma) $ corresponding to the induced subgraph $ \Gamma \setminus \{v\} $ is modular.
\end{Theorem}

From this theorem, one can show that the extended lift matroid $ M_{+}\big(K_{n}^{G}\big) $ is supersolvable for any finite group $ G $.

Zaslavsky \cite[Theorem~3.2]{zaslavsky2001supersolvable-ejoc} also characterized supersolvability of extended lift matroids as the minimal class satisfying the conditions (i) and (ii) in the following theorem.
We can consider the additional condition (iii) for modular extendedness.

\begin{Theorem}\label{main extended lift}
Let $ G $ be a finite subgroup of the additive group of a field $ \mathbb{K} $ and $ \mathcal{C}^{G}_{+} $ the minimal class of loopless gain graphs with gain group $ G $ which satisfies the following conditions.
\begin{enumerate}[$(i)$]\itemsep=0pt
\item The null graph is a member of $ \mathcal{C}^{G}_{+} $.
\item If $ \Gamma $ has a link-simplicial vertex $ v $ and $ \Gamma \setminus v \in \mathcal{C}^{G}_{+}$, then $ \Gamma \in \mathcal{C}^{G}_{+} $.
\item If there exists a decomposition $ V_{\Gamma} = V_{1} \cup V_{2} $ such that $ \Gamma[V_{1}], \Gamma[V_{2}] \in \mathcal{C}^{G}_{+}, \ E_{\Gamma} = E_{\Gamma[V_{1}]} \cup E_{\Gamma[V_{2}]} $, and $ \Gamma[V_{1} \cap V_{2}] \simeq K_{n}^{G} $ for some $ n $, then $ \Gamma \in \mathcal{C}^{G}_{+} $.
\end{enumerate}
Then for every $ \Gamma \in \mathcal{C}^{G}_{+} $ the corresponding arrangement $ \mathcal{A}_{+}(\Gamma) $ is divisionally free.
\end{Theorem}
\begin{proof}
The extended lift matroid $ M_{+}(\Gamma) $ is modularly extended by Propositions~\ref{extended lift round},~\ref{extended lift modular}, and Theorem~\ref{Zaslavsky link simplicial}.
Using Theorems~\ref{main divisionally free} and~\ref{Zaslavsky extended lift arrangement}, we can conclude that~$ \mathcal{A}_{+}(\Gamma) $ is divisionally free.
\end{proof}

\begin{Example}Let $ \bowtie $ be a signed graph described in Fig.~\ref{Fig:bowtie}.
Here we regard the gain group~$ \{\pm 1\} $ the additive group of $ \mathbb{F}_{2} $.
Then $ {\bowtie} \in \mathcal{C}^{\mathbb{F}_{2}}_{+} $ and the arrangement $ \mathcal{A}_{+}(\bowtie) $ consists of the following $ 9 $ hyperplanes
\begin{gather*}
\{z=0\},\, \{x_{1}+z=0\}, \,\{x_{2}+z=0\}, \,\{x_{1}+x_{2}=0\}, \,\{x_{1}+x_{2}=z\}, \\
\{y_{1}+z=0\}, \,\{y_{2}+z=0\},\, \{y_{1}+y_{2}=0\},\, \{y_{1}+y_{2}=z\}.
\end{gather*}
Since $ \bowtie $ has no link-simplicial vertices, $ \mathcal{A}_{+}(\bowtie) $ is not supersolvable but divisionally free.
\end{Example}

\begin{Remark}Theorem \ref{main extended lift} requires that the gain group $ G $ is finite. Hence we cannot say anything about the extended Catalan and Shi arrangements, where the gain group is $ \mathbb{Z} $, although they are free.
\end{Remark}

\subsection{Arrangements over finite fields}
Free arrangements over finite fields are investigated by Ziegler~\cite{ziegler1990matroid-totams} and Yoshinaga~\cite{yoshinaga2007free-potjasams}. Recently, Palezzato and Torielli~\cite{palezzato2019free-joac} show relations between freeness of arrangements over $ \mathbb{Q} $ and freeness of arrangements over finite fields.

Let $ \mathcal{A}(n,q) $ denote the hyperplane arrangement consisting of all hyperplanes in $ n $-dimensional vector space over the finite field $ \mathbb{F}_{q} $. The linear dependence matroid of $ \mathcal{A}(n,q) $ is known as the projective geometry $ \PG(n-1,q) $. Since the lattice $ L(\PG(n-1,q)) $ is the lattice of subspaces of the vector space $ \mathbb{F}_{q}^{n} $, every flat of $ \PG(n-1,q) $ is modular. Therefore $ \PG(n-1,q) $ and $ \mathcal{A}(n,q) $ are supersolvable.

\begin{Proposition}\label{PG round}
The linear dependence matroid on $ \mathcal{A}(n,q) $, that is, the projective geometry $ \PG(n-1,q) $ is round.
\end{Proposition}
\begin{proof}Assume that the ground set of $ \PG(n-1,q) $ is written as the union of two flats~$ F_{1}$,~$F_{2} $.
Let $ V_{1} $ and $ V_{2} $ be the subspaces of $ \mathbb{F}_{q}^{n} $ corresponding to the flats $ F_{1} $ and $ F_{2} $.
Then $ \mathbb{F}_{q}^{n} = V_{1} \cup V_{2} $. This yields $ V_{1} \subseteq V_{2} $ or $ V_{2} \subseteq V_{1} $. Thus $ \PG(n-1,q) $ is round.
\end{proof}

\begin{Proposition}[Oxley {\cite[Corollary~6.9.6]{oxley2011matroid}}]\label{PG modular}
Let $ M $ be a simple matroid representable over~$ \mathbb{F}_{q} $.
Suppose that $ X $ is a flat of $ M $ such that $ M|X $ is isomorphic to the projective geometry $ \PG(n,q) $.
Then $ X $ is modular.
\end{Proposition}

\begin{Theorem}\label{main finite field}
Let $ \mathcal{F}_{q} $ be the minimal class consisting of simple matroids representable over~$ \mathbb{F}_{q} $ satisfying the following conditions.
\begin{enumerate}[$(i)$]\itemsep=0pt
\item Every supersolvable matroids over $ \mathbb{F}_{q} $ belongs to $ \mathcal{F}_{q} $.
\item $ \mathcal{F}_{q} $ is closed under taking modular joins over the projective geometry $ \PG(n,q) $.
\end{enumerate}
If the linear dependence matroid of an arrangement $ \mathcal{A} $ over $ \mathbb{F}_{q} $ belongs to $ \mathcal{F}_{q} $, then $ \mathcal{A} $ is divisionally free.
\end{Theorem}
\begin{proof}It follows that $ \mathcal{F}_{q} $ is a subclass of $ \mathcal{ME} $ by Propositions~\ref{PG round},~\ref{PG modular}, and Theorem \ref{main ss}.
Therefore, by Theorem \ref{main divisionally free}, every arrangement in $ \mathcal{F}_{q} $ is divisionally free.
\end{proof}

\begin{Example}
Ziegler \cite[Example 4.3]{ziegler1991binary-potams} constructed a binary matroid which is not supersolvable but obtained by taking a modular join of supersolvable matroids.
We show that the binary arrangement corresponding to the matroid is divisionally free.
The arrangement $ \mathcal{A} $ is constructed as follows.
Let $ \mathcal{A}_{1} $ be an arrangement of rank $ 4 $ in $ \mathbb{F}_{2}^{4} $ consisting of the following $ 11 $ hyperplanes
\begin{gather*}
 \{z_{1}=0\},\, \{z_{2}=0\}, \,\{z_{1}+z_{2}=0\}, \\
 \{x_{1}=0\}, \,\{x_{2}=0\}, \,\{x_{1}+x_{2}=0\},\, \{x_{1}+z_{1}=0\}, \,\{x_{2}+z_{1}=0\}, \\
 \{x_{1}+x_{2}+z_{1}=0\}, \,\{x_{1}+z_{1}+z_{2}=0\},\, \{x_{2}+z_{1}+z_{2}=0\}.
\end{gather*}
Then we can prove that
\begin{gather*}
\{x_{1}=0\} \subseteq \{x_{1}=x_{2}=0\} \subseteq \{x_{1}=x_{2}=z_{1}=0\} \subseteq \{x_{1}=x_{2}=z_{1}=z_{2}=0\}
\end{gather*}
is a saturated chain consisting of modular flats in $ L(\mathcal{A}_{1}) $ by Theorem~\ref{Brylawski modular coatom} and Proposition~\ref{Brylawski modular tower}.
Therefore $ \mathcal{A}_{1} $ is supersolvable (see also \cite[Corollary~2.17]{terao1986modular-aim} and \cite[Theorem~4.3]{bjorner1990hyperplane-dcg}).
Let $ \mathcal{A}_{2} $ be an isomorphic copy of $ \mathcal{A}_{1} $ and $ \mathcal{A} $ the modular join of $ \mathcal{A}_{1} $ and $ \mathcal{A}_{2} $ over $ \{\{z_{1}=0\}, \{z_{2}=0\}, \{z_{1}+z_{2}=0\} \} \simeq \PG(1,2) $, that is, $ \mathcal{A} $ consists of the following $ 19 $ hyperplanes
\begin{gather*}
 \{z_{1}=0\}, \,\{z_{2}=0\}, \,\{z_{1}+z_{2}=0\}, \\
 \{x_{1}=0\}, \,\{x_{2}=0\},\, \{x_{1}+x_{2}=0\}, \,\{x_{1}+z_{1}=0\}, \,\{x_{2}+z_{1}=0\}, \\
 \{x_{1}+x_{2}+z_{1}=0\}, \,\{x_{1}+z_{1}+z_{2}=0\},\, \{x_{2}+z_{1}+z_{2}=0\}, \\
 \{y_{1}=0\}, \{y_{2}=0\}, \,\{y_{1}+y_{2}=0\}, \,\{y_{1}+z_{1}=0\},\, \{y_{2}+z_{1}=0\}, \\
 \{y_{1}+y_{2}+z_{1}=0\}, \,\{y_{1}+z_{1}+z_{2}=0\}, \,\{y_{2}+z_{1}+z_{2}=0\}.
\end{gather*}
By Theorem \ref{main finite field}, $ \mathcal{A} $ is divisionally free.
According to Ziegler \cite[Example~4.3]{ziegler1991binary-potams}, $ \mathcal{A} $ is not supersolvable.
\end{Example}

\subsection*{Acknowledgments}
I greatly appreciate N.~Nakashima, D.~Suyama, and M.~Torielli for valuable discussions, which provide a basis of Section~\ref{sec:applications}. I~also owe my deepest gratitude to the anonymous referees whose comments are very helpful to polish the paper.

\pdfbookmark[1]{References}{ref}
\LastPageEnding

\end{document}